\documentclass{article}

\usepackage{amsmath, amssymb, amsfonts, amsthm}
\usepackage{geometry}
\usepackage{bm}
\usepackage{graphicx}
\usepackage{textcomp}
\usepackage{mathrsfs}
\usepackage{siunitx} 
\usepackage{glossaries}
\usepackage{amsmath} 
\usepackage{amssymb}

\usepackage{zhlipsum}  
\usepackage{url}
\allowdisplaybreaks[4]

\newtheorem{theorem}{Theorem}[section]

\theoremstyle{definition}
\newtheorem{definition}[theorem]{Definition}

\theoremstyle{remark}

\newtheorem{example}[theorem]{Example}

\numberwithin{equation}{section}

\begin{document}
	
	\title {Proof of admissible weights}
	
	\vskip 16mm
	\author{
	}
	
	\date{\today}
	\maketitle
	\begin{center}
	    Hao Yu  \quad School of Mathematics and Statistics, Chongqing University, Chongqing
	\end{center}
	\begin{abstract}
		
		Admissable weight is an important tool for studying spectral invariance in operator algebra. Common  admissable weights include polynomial weights and sub exponential weights. This article mainly provides a proof that polynomial weights are permissible weights.
	\end{abstract}

	\section{Introduction } 
	The Uniform Roe algebra and Roe algebra\cite{higson1995coarse} originated from the index theory on non compact manifolds, reflecting the coarse structure of metric spaces These algebras play an important role in using the $C ^ * $algebra method to solve some geometric and topological problems (such as Novikov conjecture \cite{yu1995baum})in differential topology  Therefore, studying the spectral invariant subalgebras of consistent Roe algebras is of great significance There are many methods for studying spectral invariance. The traditional method is to start from the definition and verify spectral equality, but Hulanicki proposed a new method that no longer studies spectral equality but rather investigates whether spectral radii are equivalent. In 2007, Sun \cite{sun2007wiener} provided a definition of  admissable weights and briefly proved whether some common weight functions are admissable weights
This article will provide a detailed proof in the following content that polynomial weights are admissable weights.
	
	\section{weight function} 
	
\begin{definition}\cite{fendler2008convolution}
    If function $ w: G \longrightarrow[1, \infty) $  \quad and
		\[w(xy) \leq w(x) w(y),  \quad  \forall x,y \in G ;\]
		\[w\left(x^{-1}\right)=w(x),\quad   \forall x \in G;\]
		\[w(e)=1.\]
    
\end{definition}   

\begin{example}
    Here are some examples of classic weight functions:
	
	(i) If \quad $w_s, s \geq 0$,
	\[w_ s(x, y)=\left(1+\rho_{l}(x, y)\right)^{s},\]
	 is a weight on $G \ times G $, and this type of weight is called admissable weight
	
	(ii) If $f_{\alpha, \beta}, \alpha \in(0, \infty) $ and  $\beta \in(0,1)$ 
	\[f_{\alpha, \beta}(x, y)=\exp \left(\alpha \rho_{l}(x, y)^{\beta}\right),\]
	is a weight on $G \ times G $,this type of right is called a sub index right
\end{example}

	\begin{definition}\cite{2003Spectral}
	    	Let $G $ be a countable group and $l $ be an appropriate length function on $G $ For $\tau \in \left [1,\infty  \right )$, let $\left | B(x, \tau)\right |$ represents the number of elements in the ball $B(x,\tau) =\left \{ y\in G: \rho_{l}(x,y)<\tau \right \} $ ,We call it:
	
	The group G is polynomial growing. If C exists and d>0, then
	\begin{equation}\label{ud}
		\left | B(x,\tau)\right |\le C\tau ^{d} \quad\text{for all} ~ y\in G~\text{and}~ \tau \ge 1,
	\end{equation}

	\end{definition}
\section{admissable weight}	
\begin{definition}
		Let $1\le p,r\le \infty$. We say that a weight $\omega $ is $\left (p,r  \right ) $-{\it admissible} if there exist another weight $v$ and two positive constants $D\in \left ( 0,\infty  \right ) $ and
		 $ \theta \in \left ( 0,1 \right ) $ such that
		\begin{equation}\label{w1}
			w\left ( x,y \right ) \le  D\left ( w\left ( x,z \right ) v\left ( z,y \right ) +v\left ( x,z \right ) w\left ( z,y \right )  \right )\quad\text{for~all}~ x,y,z\in G,
		\end{equation}
		\begin{equation}\label{w22}
			\sup_{x\in G}\left \| (vw^{-1})\left ( x,\cdot \right ) \right \|_{p^{_{'}}} +\sup_{y\in G}\left \|( vw^{-1})\left ( \cdot,y \right ) \right \| _{p^{_{'}}}\le D,
		\end{equation}
		and
		\begin{equation}\label{w2}
			\inf _{\tau >0}a_{r' }\left ( \tau  \right ) +b_{p' }\left ( \tau  \right )t                                                                                                                                                                                                                                                                                                                                                                                                                                                                                                                                                                                                                                                                                                                                                                                                                                                                                                                                                                                                                                                                                                                                                                                                                                                                                                                                                                                                                                                                                                                                                                                                                                                                                                                                                                                                                                                                                                                                                                                                                                                                                                                                                                                                                                                                                                                                                                                                                                                                                                                                                                                                                                                                                                                                                                                                                                                                                                                                                                                                                                                                                                                                                                                                                                                                                                                                                                                                                                                                                                                                                                                                                                                                                                                                                                                                                                                                                                                                                                                                                                                                                                                                                                                                                                                                                                                                                                                                                                                                                                                                                                                                                                                                                                                                                                                                                                                                                                                                                                                                                                                                                                                                                                                                                                                                                                                                                                                                                                                                                                                                                                                                                                                                                                                                                                                                                                                                                                                                                                                                                                                                                                                                                                                                                                                                                                                                                                                    
		\le Dt^\theta \quad \text{for~all}~ t\ge 1,
		\end{equation} 
		where ${p}' =p/\left(p-1 \right)$, $r' =r/\left ( r-1 \right )$,
		\begin{equation}\label{w3}
			a_{r'}(\tau)=\sup_{x\in G}\left \| v\left ( x,\cdot \right )\chi _{B\left ( x,\tau  \right )\left ( \cdot  \right )  }  \right \|_{r^{_{'}}} +\sup_{y\in G}\left \| v\left ( \cdot,y \right )\chi _{B\left (y, \tau  \right )\left ( \cdot  \right )  } \right \|_{r^{_{'}}},
		\end{equation}
		\begin{equation}\label{w end}
			b_{p' }(\tau)=\sup_{x\in G}\left \| (vw^{-1})\left ( x,\cdot \right )\chi _{X\setminus B\left ( x,\tau  \right )\left ( \cdot  \right )  }  \right \|_{p^{_{'}}} +\sup_{y\in G}\left \|( vw^{-1})\left ( \cdot,y \right )\chi _{X\setminus B\left (y, \tau  \right )\left ( \cdot  \right )  } \right \|_{p^{_{'}}},
		\end{equation}
		$\chi _{E}$ is the characteristic function on the set $E$, and $\left \| \cdot  \right \| _{p}$ is the norm on $\ell^{p}$,  the space of all $\mathnormal{p}$-summable functions on $G$.
		
		The technical assumption on the weight $w$, $(p, r)$-admissibility, plays very important role in our  results,.
	\end{definition}

\section{Proof of admissable Weights for Polynomial Weights}

\begin{theorem}
		\label{poly}
		Let $G$ be a countable discrete group with a  metric $\rho_{l } $ and suppose $G$ has polynomial growth.  Define $w(x,y)=(1+\rho_{l}(x,y))^s$ as a left-invariant weight on $G\times G $. Then there exists another weight $v$ such that (\ref{w1})-(\ref{w2}), indicating that $w$ is a $\left (p,r \right)$-admissible weight.
	\end{theorem}
	\begin{proof}
				We define the minimal rate of polynomial growth $d(G,\rho_l)$  as follows
		\begin{align} \label{defd}
			d(G,\rho_l)=\inf\left \{ d:(\ref{ud})\text{ holds for some positive } C \right \}.
		\end{align}
	 For $s\ge 0$, the weight $w$ and $v$ are defined as follows
		\begin{equation*}
			w(x,y)=(1+\rho_{l}(x,y))^s\quad \text{for~all}~ x,y\in G,
		\end{equation*}
			\begin{equation*}
			v(x,y)=(1+\rho_{l}(x,y))^0=1\quad \text{for~all}~ x,y\in G.
		\end{equation*}
		Then, we have
		\begin{align*}
			w(x,y)=(1+\rho_{l}(x,y))^s&\le (1+\rho_{l}(x,z)+\rho_{l}(z,y))^s\\
			&\le(1+\rho_{l}(x,z))^s +(1+\rho_{l}(z,y))^s\\
			&=w(x,z)\cdot 1+w(z,y)\cdot1\\
			&=w(x,z)v(z,y)+v(x,z)w(z,y),
		\end{align*}
		for all $x,y,z\in G$ with $D=1$. Thus, we have the weights $w$ and $v$ satisfy (\ref{w1}).

	Let $s>\frac{1}{p'}d(G,\rho_l)$. Using (\ref{ud}) and (\ref{defd}), 
    
    we  obtain the following inequality for $\sum_{\rho_l(x,y)\ge \tau} w^{-p'}(x,y)$:
	
	\begin{align*}
			\sum _{\rho_l(x,y)\ge \tau }w^{-p'}(x,y)
			&=\sum_{j=1}^{\infty } \sum _{2^{j-1}\tau \le\rho_l (x,y)< 2^{j}\tau}w^{-p'}(x,y) \notag\\ 
			&\le \sum_{j=1}^{\infty }(1+2^{j-1}\tau )^{-p's}\left | B(x,2^j\tau) \right |\le C \sum_{j=1}^{\infty }(1+2^{j-1}\tau )^{-p's}\cdot \left ( 2^{j}\tau  \right )^{d(G,\rho_l)} \notag \\ 
			&\le C 2^{p's}\cdot \tau^{-p's+d(G,\rho_l)} \sum_{j=1}^{\infty }2^{j(d(G,\rho_l)-p's)} 
			 \notag \\  
			&\le C_{\epsilon }\tau ^{-(p's-d(G,\rho_l))}<\infty\quad \text{for all}~ x\in G~ \text{and} ~\tau \ge 1,
		\end{align*}
		 where $C$ is a positive constant independent of $x\in G$.
		Then, we get
		\begin{align*}
				\sup_{x\in G}\left \| vw^{-1}(x,\cdot ) \right\| _{p'}&=\sup_{x\in G}\Big ( \sum _{y\in G}\left ( w^{-p'}(x,y) \right )  \Big)^\frac{1}{p'} \\
			&= \sup_{x\in G}\Big(\sum _{0<\rho_{l}(x,y)< \tau }w^{-p'}(x,y)\Big)^\frac{1}{p'}  +\sup_{x\in G}\Big(\sum _{\rho_{l}(x,y))\ge \tau }w^{-2}(x,y)\Big)^\frac{1}{p'}\\
			&\le C+\sup_{x\in G}\Big(\sum _{\rho_{l}(x,y))\ge \tau }w^{-p'}(x,y)\Big)^\frac{1}{p'}<\infty.
		\end{align*}
		Similarly, we obtain
		$$	\sup_{x\in G}\left \| vw^{-1}(\cdot,y ) \right\| _{p'}<\infty.$$
		Therefore, we have
		\[\sup_{x\in G}\left \|v w^{-1}(x,\cdot ) \right \| _{p'}+\sup_{y\in G}\left \| vw^{-1}(\cdot,y ) \right \| _{p'}<\infty,\]
		which implies the weights $w$ and $v$ satisfy (\ref{w22}).
		
		We claim that
		\begin{align*}
			\inf_{\tau\ge 1}\left( a_{r'}(\tau ) +b_{p'}(\tau ) \cdot t\right)
			&\le Ct^{\theta},
		\end{align*}
		for all $t\ge1$ and $\theta=\frac{d(G,\rho_l) }{d(G,\rho_l)+(s-\frac{d(G,\rho_l)}{2})(\frac{r}{r-1} ) } <1 $.
		
			Firstly, we have
		\begin{align*}
			\sup_{x\in G}\left \| v\left ( x,\cdot \right )\chi _{B\left ( x,\tau  \right )\left ( \cdot  \right )  }  \right \|_{r'}&=\sup_{x\in G}\Big[ \sum_{y\in G}  ( 	\chi _{B\left ( x,\tau  \right )\left ( y \right ) } )^{r'} \Big]^\frac{1}{r'}\\
			&\le\sup_{x\in G}(\left| B(x,\tau)\right| )^{\frac{1}{r'}}\\
			&\le C_{1} \tau ^{\frac{d(G,\rho_l)}{r'}}.
		\end{align*}
		Similarly, we have
		\begin{align*}
\sup_{x\in G}\left \| (vw^{-1})\left ( x,\cdot \right )\chi _{X\setminus B\left ( x,\tau  \right )\left ( \cdot  \right )  }  \right \|_{p'}& = \sup_{x\in G}(\sum _{y\in G,\rho_l(x,y)\ge \tau }(1+\rho_{l}(x,y))^{-sp'})^{\frac{1}{p'}} \\
&\le\sup_{x\in G} (\sum _{y\in G,\rho_l(x,y)\ge \tau }w^{-p'}(x,y))^{\frac{1}{p'}}\\
& \le C_{1}  \tau ^{-(s-\frac{d(G,\rho_l)}{p'}    )}.
\end{align*}
		Then, we get
		\begin{align*}
			\inf_{\tau\ge 1}\left( a_{r'}(\tau ) +b_{p'}(\tau ) \cdot t\right)
			&\le 2\inf_{\tau\ge 1} \left ( C_{1}\tau^\frac{d(G,\rho_l)}{r'} +C_{1}\tau ^{-(s-\frac{d(G,\rho_l)}{p' }) }\cdot t \right )\\
            &=\inf_{\tau\ge 1} \left ( C\tau^\frac{d(G,\rho_l)}{r'} +C\tau ^{-(s-\frac{d(G,\rho_l)}{2 }) }\cdot t \right )
		\end{align*}
		where ${p}' =p/\left(p-1 \right)$, $r' =r/\left ( r-1 \right )$,$1\le r \le \infty$.. 
		
		 Indeed, let $\tau=t^{\alpha},\alpha\ge 0$, it is enough to prove
	\begin{align}
		t^{\frac{\alpha d(G,\rho_l)}{r'} }&\le Ct^{\frac{d(G,\rho_l) }{d(G,\rho_l)+(s-\frac{d(G,\rho_l)}{p'})(\frac{r}{r-1}) } }\label{e3.22}\\ 
		t^{-\alpha (s-\frac{d(G,\rho_l)}{p'})} \cdot t &\le Ct^{\frac{d(G,\rho_l) }{d(G,\rho_l)+(s-\frac{d(G,\rho_l)}{p'})(\frac{r}{r-1}) }}\label{e3.23}.
	\end{align}

	By setting$\alpha=\frac{1}{(1-\frac{1}{r})d(G,\rho_l)+(s-\frac{d(G,\rho_l)}{p'}) }$, we obtain inequalities (\ref{e3.22}) and (\ref{e3.23}).
	
Hence, we have
			\begin{align*}
			\inf_{\tau\ge 1}\left( a_{r'}(\tau ) +b_{p'}(\tau ) \cdot t\right)
			&\le\inf_{\tau\ge 1} \left ( C\tau^\frac{d(G,\rho_l)}{r'} +C\tau ^{-(s-\frac{d(G,\rho_l)}{p' }) }\cdot t \right ) \\
			&\le Ct^{\frac{d(G,\rho_l) }{d(G,\rho_l)+(s-\frac{d(G,\rho_l)}{p'})(\frac{r}{r-1}) }}=Ct^{\theta},
		\end{align*}
		which means the weights $w$ and $v$ satisfy (\ref{w2}).	
\end{proof}

\bibliographystyle{abbrv}
	\bibliography{main}

\end{document}